\documentclass[11pt,reqno]{amsart}
\usepackage[utf8]{inputenc}
\usepackage{amsmath,amssymb,amsthm, mathrsfs}
\usepackage{indentfirst}
\usepackage{comment,relsize,braket}
\usepackage{enumerate,multicol}
\usepackage[numbers]{natbib}
\usepackage{graphicx}
\usepackage{color}
\usepackage[all]{xy}
\usepackage{fullpage}
\usepackage{tikz}
\usepackage[pagebackref,colorlinks, linkcolor=blue]{hyperref}
\usepackage[top=1in, bottom=1.3in, left=1in, right=1in]{geometry}

\usepackage{fancyhdr}
\newtheorem{Theorem}{Theorem}[section]

\newtheorem{Corollary}[Theorem]{Corollary}

\newtheorem{Example}[Theorem]{Example}

\newtheorem{Definition}[Theorem]{Definition}

\numberwithin{equation}{section}

\everymath{\displaystyle}

 \def\NN{{\mathbb N}}

 \def\HH{{\mathbb H}}

\def\frk{\mathfrak}  
   
 \def\nn{{\frk n}} \def\mm{{\frk m}}
\def\R{{\mathcal R}}  
\def\opn#1#2{\def#1{\operatorname{#2}}} % to make operators
\opn\chara{char} \opn\length{\ell} \opn\pd{pd} \opn\rk{rk}
\opn\projdim{proj\,dim} \opn\injdim{inj\,dim}
\opn\rank{rank} \opn\depth{depth} \opn\grade{grade} 
\opn\hei{ht} \opn\embdim{emb\,dim}\opn\codim{codim}
\opn\Tr{Tr} \opn\bigrank{big\,rank}
\opn\superheight{superheight} \opn\lcm{lcm}
\opn\rdim{rdim} \opn\trdeg{tr\,deg} \opn\reg{reg}  \opn\lreg{lreg} 
\opn\ini{in} \opn\lpd{lpd} \opn\size{size} \opn{\mult}{mult}
\opn\div{div} \opn\Div{Div} \opn\cl{cl} \opn\Cl{Cl}
\opn\Spec{Spec} \opn\Supp{Supp} \opn\supp{supp} 
\opn\Sing{Sing} \opn\Ass{Ass} \opn\Min{Min}
\opn\Proj{Proj} \opn{\Max}{Max} \opn{\Assh}{Assh}
\opn\Ann{Ann} \opn\Rad{Rad} \opn\Soc{Soc}
\opn\Syz{Syz} \opn\Im{Im} \opn\Ker{Ker} \opn\Coker{Coker}
\opn\Am{Am} \opn\Hom{Hom} \opn\tor{Tor} \opn\Ext{Ext}

\opn{\h}{{\bf h}}
\opn{\HH}{\text{H}}

\begin{document}

\title[]{Generalized Hilbert-Kunz function of the Rees algebra of the face ring of a simplicial complex}

%  author one information
\author[]{Arindam Banerjee}
\address{Ramakrishna Mission Vivekananda Educational and Research Institute, Belur, India}
\email{123.arindam@gmail.com}

%  author two information
\author[]{Kriti Goel}
\address{Indian Institute of Technology Bombay, Mumbai, India 400076}
\email{kritigoel.maths@gmail.com}

%  author three information
\author[]{J. K. Verma}
\address{Indian Institute of Technology Bombay, Mumbai, India 400076}
\email{jkv@math.iitb.ac.in}

%\date{\today}

%    The 2010 edition of the Mathematics Subject Classification is
%    the current definitive version.

\begin{abstract}
	Let  $R$ be the face ring of a simplicial complex of dimension $d-1$ and 
	${\mathcal R}(\nn)$ be the Rees algebra of the maximal homogeneous ideal $\nn$ of $R.$ We show that the generalized Hilbert-Kunz function $HK(s)=\ell({\R}(\nn)/(\nn, \nn t)^{[s]})$  is given by a polynomial for all large $s.$ We calculate it in many examples and also provide a Macaulay2 code for computing $HK(s).$ 
\end{abstract}

\maketitle

\centerline{\em Dedicated to Roger Wiegand and Silvia Wiegand on the occasion of their  $150^{th}$ birthday}
%\medskip 
\section{Introduction}
The objective of this paper is to find the generalized Hilbert-Kunz function 
of the maximal homogeneous ideal of the Rees algebra of the maximal homogeneous
ideal of the face ring of a simplicial complex. The Hilbert-Kunz functions of the Rees algebra, associated graded ring  and the extended Rees algebra have been studied by K. Eto and K.-i. Yoshida in \cite{ey} and by K. Goel, M. Koley and J. K. Verma in \cite{gkv}. 

In order to recall one of the main results of Eto and Yoshida, we set up some notation first. Let $(R,\mm)$ be a
Noetherian local ring of dimension $d$ and of prime characteristic $p.$  Let $q=p^e$ where $e$ is a non-negative integer. The $q^{th}$ Frobenius power of an ideal $I$ is defined to be $I^{[q]}=(a^q\mid a\in I).$ Let $I$ be an $\mm$-primary ideal. The Hilbert-Kunz function of $I$ is the function $HK_{I}(q)=\ell(R/I^{[q]}).$ This function, for $I=\mm$, was introduced by E. Kunz in \cite{k} who used it to characterize regular local rings. The Hilbert-Kunz multiplicity of an $\mm$-primary ideal $I$ is defined as $e_{HK}(I)=\lim_{q\to \infty}\ell(R/I^{[q]})/q^d.$ It was introduced by P. Monsky in \cite{pm}. We refer the reader to an excellent survey paper  of C. Huneke \cite{h} for further details. Let $e(I)$ denote the Hilbert-Samuel multiplicity of $R$ with respect to $I.$ We write $e(\mm)=e(R)$
for a local ring $(R,\mm)$ and $e_{HK}(R)=e_{HK}(\nn)$ where $\nn$ is the unique maximal homogeneous ideal of a graded ring $R.$

Eto and Yoshida calculated the Hilbert-Kunz multiplicity of various blowup algebras of an ideal under certain conditions.  Put  $c(d) = (d/2)+ d/(d+1)!.$ They proved the following. 

\begin{Theorem}
	Let $(R,\mm)$ be a Noetherian local ring of prime characteristic $p > 0$ with $d =\dim R \geq 1.$ Then for any $\mm$-primary ideal $I$, we have
	\[ e_{HK}(\R(I)) \leq c(d) \cdot e(I). \]
	Moreover, equality holds if and only if $e_{HK}(R)=e(I).$ When this is the case, $e_{HK}(R)=e(R)$ and $e_{HK}(I)=e(I).$ Here $e_{HK}(R)=e_{HK}(\mm)$ and $e(R)=e(\mm).$
\end{Theorem}
It is natural to ask if there is a formula for the Hilbert-Kunz function and the Hilbert-Kunz multiplicity of the maximal homogeneous ideal $(\mm,It)$ of the Rees algebra ${\mathcal R}(I)=\oplus_{n=0}^\infty I^nt^n$  where $I$ is an $\mm$-primary ideal, in terms of invariants of the ideals $\mm$ and $I.$ 
In this paper we answer this question for the Rees algebra of the maximal homogeneous  ideal of the face ring of a simplicial complex.  In fact, we find its generalized Hilbert-Kunz function. The generalized Hilbert-Kunz function was introduced by Aldo Conca in \cite{conca}. Let $(R,\mm)$ be a $d$-dimensional Noetherian local  (resp. standard graded)  ring with maximal (resp. maximal homogeneous) ideal $\mm$  and $I$ be an $\mm$-primary (resp. a graded $\mm$-primary) ideal. Fix a set of generators of $I$, say $I=(a_1, a_2,\ldots, a_g).$ We choose these as homogeneous elements in case $R$ is a graded ring. Define the $s^{th}$ Frobenius power of $I$ to be the ideal $I^{[s]}=(a_1^s, a_2^s, \ldots, a_g^s).$ The generalized Hilbert-Kunz function of $I$ is defined as $HK_I(s)=\ell(R/I^{[s]}).$ The generalized Hilbert-Kunz multiplicity is defined as $\lim\limits_{s \rightarrow \infty} HK_I(s)/s^d$, whenever the limit exists.

We now describe the contents of the paper. Let $\Delta$ be a simplicial complex of dimension $d-1.$ Let $k$ be any field, $k[\Delta]$ denote the face ring of $\Delta$ and $\nn$ be its maximal homogeneous ideal. Let $\R(\nn)=\oplus_{n=0}^\infty \nn^nt^n$ be the Rees algebra of $\nn.$  In section 2, we collect some preliminaries required for estimation of the asymptotic reduction number in terms of the $a$-invariants of local cohomology modules and Hilbert-Samuel polynomial of the maximal  homogeneous ideal of the face  ring of a simplicial complex. Section 3 is devoted to the  computation of the generalized Hilbert-Kunz function $HK_{(\nn, \nn t)}(s)$, where $(\nn, \nn t)$ is the maximal homogeneous ideal of the Rees algebra ${\mathcal R}(\nn).$  We also estimate an upper bound on the postulation number of $HK_{(\nn, \nn t)}(s)$ in terms of $a$-invariants of the local cohomology modules. This enables us to explicitly calculate the generalized Hilbert-Kunz function for the Rees algebra in several examples such as the edge ideal of a complete bipartite  graph, the real projective plane and a few other examples of simplicial complexes. We have implemented the formula for the Hilbert-Kunz function in an algorithm written in the language of Macaulay2.\\

\textbf{Acknowledgement:} We thank the  anonymous referee for his valuable suggestions.

\section{Preliminaries}

In this section we gather some results which we shall use in the later sections. 

Let $R$ be a ring and $I$ be an $R$-ideal.  Let $G(I)=\oplus_{n\geq 0} {I^n}/{I^{n+1}}$ be the associated graded ring of $I$. An ideal $J \subseteq I$ is called a reduction of $I$ if $JI^n = I^{n+1}$, for all large $n.$ A minimal reduction of $I$ is a reduction of $I$ minimal with respect to inclusion. For a minimal reduction $J$ of $I$, we set $$r_J(I) = \min \{ n \mid  I^{m+1} = JI^m \text{ for all } m \geq n \}.$$ 
The reduction number of $I$ is defined as 
\[ r(I) = \min\{ r_J(I) \mid J \text{ is a minimal reduction of } I \}. \]

Let $(R,\mm)$ be a $d$-dimensional local ring and $I$ be an $\mm$-primary ideal. It is known that $H_I(n) := \ell(R/I^n)$ is a polynomial function of $n$ of degree $d$, for large $n$. In particular, there
exists a polynomial $P_I(x) \in \mathbb{Q}[x]$ such that $H_I(n)=P_I(n)$ for all large $n$. The postulation number of $I$ is defined as 
\[ n(I) = \max\{ n \mid H_I(n) \neq P_I (n) \}. \]

Let $M$ be a finitely generated $R$-module. We define $a_i(M)$ to be $\max\{u \in \mathbb{Z} \mid [H_{\mm}^{i}(M)]_u \neq 0\}$ if $H_{\mm}^{i}(M) \neq 0$, and  $-\infty$ otherwise. We shall use the following results to estimate the reduction number of powers of an ideal. 

\begin{Theorem}[{\cite[Corollary 2.21]{marleyThesis}}] \label{Marley}
	Let $(R, \mm)$ be a $d$-dimensional Cohen-Macaulay local ring with infinite residue field and $I$ be an $\mm$-primary ideal such that $\grade(G(I)_+) \geq d - 1.$ Then for $k \geq 1$, 
	\[ r(I^k) = \left\lfloor \frac{n(I)}{k} \right\rfloor + d. \]
\end{Theorem}

\begin{Theorem}[{\cite[Theorem 2.1]{hoa}}] \label{Hoa}
	Let $(R,\mm)$ be a Noetherian local ring and let $I \subseteq \mm$ be an $R$-ideal. Then $r_J(I^n)$ is independent of $J$ and stable if $n$ is large. In particular, for all $n > \max\{ |a_i(G(I))| \colon a_i(G(I)) \neq - \infty \}$, we get
	\begin{align*}
	r_J(I^n) = 
	\begin{cases}
	s & \text{ if } a_s(G(I)) \geq 0, \\
	s-1 & \text{ if } a_s(G(I)) <0,
	\end{cases}
	\end{align*}
	where $J$ is any minimal reduction of $I^n$ and $s$ is the analytic spread of $I.$
\end{Theorem}

Let $S$ be a $d$-dimensional Cohen-Macaulay local ring and let $I$ be a parameter ideal. Fix $s \in \mathbb{N}.$ For a fixed set of generators of $I$, define functions 
\[F(n) := H_I(I^{[s]},n) = \ell_S\left( \dfrac{I^{[s]}}{I^{[s]}I^n} \right) \text{ \ and \ } H(n) := H_I(S,n) = \ell_S\left( \dfrac{S}{I^n} \right) = e(I) \binom{n+d-1}{d} \] 
for all $n.$ Note that if $S$ is $1$-dimensional, then $F(n)=H(n)$ for all $n$ and for all $s$. In \cite{gkv}, the authors prove that the function $F(n)$ is a piecewise polynomial in $n.$

\begin{Theorem} [{\cite[Theorem 3.2]{gkv}}]  \label{F(n)}
	Let $S$ be a $d$-dimensional Cohen-Macaulay local ring and $I$ be a parameter ideal. Let $d \geq 2.$ For a fixed $s \in \mathbb{N}$,
	\begin{align*}
	F(n) = 
	\begin{cases}
	d \ H(n) &  \text{if } 1 \leq n \leq s, \\
	\sum_{i=1}^{d-1} (-1)^{i+1} \binom{d}{i} H(n-(i-1)s) & \text{ if }  s+1 \leq n \leq (d-1)s-1, \\
	H(n+s) - s^d e(I) & \text{ if } n \geq (d-1)s.
	\end{cases}
	\end{align*}
\end{Theorem}

Let $\Delta$ be a $(d-1)$-dimensional simplicial complex on the vertex set $[n]=\{1,2,\ldots,n\}.$ Let $k$ be a  field and $S=k[x_1, x_2,\ldots,x_n]$ be the polynomial ring over $k.$ For $F\subset [n],$ we put $x^F=\prod_{i\in F}x_i.$ The ideal of $\Delta$ is $I_{\Delta}=(x^F\mid F\notin \Delta)S.$
The face ring or the Stanley-Reisner ring of $\Delta$ is the ring $k[\Delta]:=S/I_{\Delta}.$  Let $\nn$ denote the unique maximal homogeneous ideal of $k[\Delta].$ The $f$-vector of $\Delta$ is $f(\Delta)=(f_{-1}, f_0,\ldots, f_{d-1}),$ where $f_{-1}=1$ and $f_i$ is the number of $i$-dimensional faces of $\Delta$, for $i=0,1,\ldots, d-1.$ The Hilbert series of $k[\Delta]$ is the formal power series $H(k[\Delta],t)=\sum_{i=0}^\infty \dim_k k[\Delta]_it^i$ where $k[\Delta]_i$ is the graded component of $k[\Delta]$ consisting of  homogeneous elements of degree $i$ in $k[\Delta].$ 

 Stanley showed the following:

\begin{Theorem}[{\cite[Theorem 1.4]{Stanley}}]\label{Stanley}
$H(k[\Delta],t)=\sum_{i=-1}^{d-1} \frac{f_i t^{i+1}}{(1-t)^{i+1}}$
\end{Theorem}

By taking the lcm of the denomenators we write the Hilbert series of $k[\Delta]$ as the rational function $H(k[\Delta],t)=(h_0+h_1t+\cdots+h_dt^d)/(1-t)^d.$ The vector $(h_0,h_1,\ldots, h_d)$ is called the $h$-vector of $\Delta$. Let $h^{(i)}(t)$ denote the $i^{th}$ derivative of $h(t) = h_0+h_1t+\cdots+h_dt^d.$
\begin{Theorem}[{\cite[Theorem 6.2]{gmv}}] \label{faceringpoly} 
	The Hilbert-Samuel function $\ell(k[\Delta]/\nn^s),$ for all $s \geq 1,$ is given by
 	\[ \ell\left( \frac{k[\Delta]}{\nn^s} \right) = \sum_{i=0}^d (-1)^i\frac{h^{(i)}(1)}{i!}\binom{s-1+d-i}{d-i}. \]
\end{Theorem}

The following result of Conca computes the generalized Hilbert-Kunz function of $k[\Delta]$.

\begin{Theorem}[{\cite[Remark 2.2]{conca}}] \label{conca}
	For $s \geq 1$, the generalized Hilbert-Kunz function of $k[\Delta]$ is given by the equation
	\[ \ell\left(\frac{k[\Delta]}{\nn^{[s]}}\right) = \sum_{i=0}^{d} f_{i-1}(s-1)^i. \]
\end{Theorem}

%\begin{proof}
%	Observe that $\ell(R/\nn^{[s]}) = |V|$, where 
%	\[ V = \{ {\bf a} = (a_1,\ldots,a_r) \in \NN^r \mid 0 \leq a_i \leq s-1 \text{ and } \Supp({\bf a}) \in \Delta \}. \]
%	Therefore, for $s \geq 1,$
%	\[
%	\ell\left(\frac{R}{\nn^{[s]}}\right) 
%	= \sum_{F \in \Delta} |\{ {\bf a} \in V \mid \Supp({\bf a}) \in F \}| 
%	= \sum_{F \in \Delta} (s-1)^{|F|} 
%	= \sum_{i=0}^{d} f_{i-1} (s-1)^i.
%	\]
%\end{proof}

\section{The generalized Hilbert-Kunz function of  $(\nn, \nn t)$ }

Let $S = k[x_1,\ldots,x_r]$ be a polynomial ring in $r$ variables over a field $k$ and let $\mm = (x_1,\ldots,x_r)$ denote the maximal homogeneous ideal of $S.$ Let $P_j$, for $j=1,\ldots,\alpha$ and $\alpha \geq 2$, be distinct $S$-ideals generated by subsets of $\{ x_1,\ldots,x_r \}.$ Let $I = \cap_{j=1}^{\alpha} P_j$ and $R=S/I.$ Let $\nn = \mm/I$ denote the maximal homogeneous ideal of $R.$ 

In this section, we show that the generalized Hilbert-Kunz function of the maximal homogeneous ideal $(\nn,\nn t)$ of the Rees algebra $\R(\nn)$ of $R$ is a polynomial for large $s.$ We begin by proving that for $s,n \in \NN$, $\ell_S(S/I+\mm^{[s]}\mm^n)$ is a piecewise polynomial in $s$ and $n.$ First we prove the following result which is a consequence of Theorem \ref{F(n)}.

\begin{Corollary} \label{corF(n)}
	Let $S=k[x_1,\ldots,x_d]$ be a polynomial ring in $d$ variables over a field $k.$ Let $\mm = (x_1,\ldots,x_d)$ be its maximal homogeneous ideal. Let $s,n \in \NN.$ \\
	{\rm(1)} If $d=1$, then $\ell \left(\frac{S}{\mm^{[s]}\mm^n}\right) = s+n$ for all $s,n\geq 0.$ \\
	{\rm(2)} If $d=2$, then
	\begin{align*}
	\ell \left(\frac{S}{\mm^{[s]}\mm^n}\right) = 
	\begin{cases}
	s^2 + n^2+n &  \text{if } 1 \leq n \leq s, \\[1mm]
	\binom{n+s+1}{2} & \text{ if } n \geq s.
	\end{cases}
	\end{align*}
	{\rm(3)} If $d \geq 3$, then
	\begin{align*}
	\ell \left(\frac{S}{\mm^{[s]}\mm^n}\right) = 
	\begin{cases}
	s^d + d \binom{n+d-1}{d} &  \text{if } 1 \leq n \leq s, \\[4mm]
	s^d + \sum_{i=1}^{d-1} (-1)^{i+1} \binom{d}{i} \binom{n-(i-1)s+d-1}{d}  & \text{ if }  s+1 \leq n \leq (d-1)s-1, \\[4mm]
	\binom{n+s+d-1}{d} & \text{ if } n \geq (d-1)s.
	\end{cases}
	\end{align*}
\end{Corollary}

\begin{proof}
	Let $s,n \in \NN.$ If $d=1$, then $S=k[x]$ and $\mm=(x)$ implying that
	\[ \ell\left(\frac{S}{\mm^{[s]}\mm^n}\right) = \ell\left(\frac{k[x]}{(x^{s+n})}\right) = s+n. \]
	Let $d \geq 2.$ Since 
	\[ \ell\left(\frac{S}{\mm^{[s]}\mm^n}\right) = \ell\left(\frac{S}{\mm^{[s]}}\right) + \ell\left(\frac{\mm^{[s]}}{\mm^{[s]}\mm^n}\right) \]
	and $\ell(S/\mm^{[s]}) = s^d$, the result follows from Theorem \ref{F(n)}.
\end{proof}

Let $T = \oplus_{n \geq 0} \ T_n$ be a Noetherian graded ring, where $T_0$ is an Artinian ring. Let $M = \oplus_{n \geq 0} \ M_n$ be a finitely generated graded $T$-module. Then $\ell_{T_0}(M_n) < \infty.$ The Hilbert series $H(M,\lambda)$ of $M$ is defined by $H(M,\lambda) = \sum_{n \geq 0} \ell_{T_0}(M_n) \lambda^n.$ 

\begin{Theorem} \label{thm1:alpha2}
	Let $T$ be a standard graded Artinian ring and let $I_1, \ldots, I_\alpha$, for $\alpha \geq 2$, be  homogeneous $T$-ideals. Let $I = \cap_{i=1}^{\alpha} I_i.$  Then
	\[ H\left(\frac{T}{I}, \lambda \right) = \sum_{i=1}^{\alpha} H\left(\frac{T}{I_i}, \lambda \right) - \sum_{\substack{i,j=1\\i<j}}^{\alpha} H\left(\frac{T}{I_i+I_j}, \lambda \right)+ \cdots +(-1)^{\alpha-1} H\left(\frac{T}{\sum_{i=1}^{\alpha}I_i}, \lambda \right). \]
\end{Theorem}

\begin{proof}
	 Apply induction on $\alpha.$ Let $\alpha=2.$ Consider the following short exact sequence
	\begin{align*}
	0 \longrightarrow \frac{T}{I_1 \cap I_2} \longrightarrow \frac{T}{I_1} \bigoplus \frac{T}{I_2} \longrightarrow \frac{T}{I_1+I_2} \longrightarrow 0. 
	\end{align*}
	Then
	\begin{align*}
	H\left(\frac{T}{I}, \lambda \right) = H\left(\frac{T}{I_1 \cap I_2}, \lambda \right) 
	&= H\left(\frac{T}{I_1}, \lambda \right) + H\left(\frac{T}{I_2}, \lambda \right) - H\left(\frac{T}{I_1+I_2}, \lambda \right). 
	\end{align*}
	Let $\alpha>2$ and consider the short exact sequence
	\begin{align*}
	0 \longrightarrow \frac{T}{\cap_{i=1}^{\alpha} I_i} \longrightarrow \frac{T}{\cap_{i=1}^{\alpha-1} I_i} \bigoplus \frac{T}{I_\alpha} \longrightarrow \frac{T}{\cap_{i=1}^{\alpha-1}I_i+I_\alpha} \longrightarrow 0. 
	\end{align*}
	Using induction hypothesis, it follows that
	\begin{align*}
	H\left(\frac{T}{\cap_{i=1}^{\alpha} I_i}, \lambda \right) 
	&= H\left(\frac{T}{\cap_{i=1}^{\alpha-1} I_i}, \lambda \right) + H\left(\frac{T}{I_\alpha}, \lambda \right) - H\left(\frac{T}{\cap_{i=1}^{\alpha-1}I_i+I_\alpha}, \lambda \right) \\
	&= \sum_{i=1}^{\alpha-1} H\left(\frac{T}{I_i}, \lambda \right) - \sum_{\substack{i,j=1\\i<j}}^{\alpha-1} H\left(\frac{T}{I_i+I_j}, \lambda \right) + \cdots +(-1)^{\alpha-2} H\left(\frac{T}{\sum_{i=1}^{\alpha-1}I_i}, \lambda \right) \\
	&+ H\left(\frac{T}{I_\alpha}, \lambda \right) - \sum_{i=1}^{\alpha-1} H\left(\frac{T}{I_i+I_\alpha}, \lambda \right) + \sum_{\substack{i,j=1\\i<j}}^{\alpha-1} H\left(\frac{T}{I_i+I_j+I_\alpha}, \lambda \right)+ \cdots  \\
	&\hspace{6.5cm}+(-1)^{\alpha-1} H\left(\frac{T}{\sum_{i=1}^{\alpha-1}I_i+I_\alpha}, \lambda \right). 
%	&= \sum_{i=1}^{\alpha} H\left(\frac{S}{I_i+J}, \lambda \right) - \sum_{\substack{i,j=1\\i<j}}^{\alpha} H\left(\frac{S}{I_i+I_j+J}, \lambda \right)+ \cdots +(-1)^{\alpha-1} H\left(\frac{S}{\sum_{i=1}^{\alpha}I_i+J}, \lambda \right).
	\end{align*}
	Rearranging the terms gives the required result.
\end{proof}

\begin{Corollary} \label{cor1:alpha2}
	Let $S = k[x_1,\ldots,x_r]$ be a polynomial ring in $r$ variables over a field $k$ and let $\mm = (x_1,\ldots,x_r)$ be the maximal homogeneous ideal of $S.$ Let $P_1, \ldots, P_\alpha$, for $\alpha \geq 2$, be distinct $S$-ideals generated by subsets of $\{x_1,\ldots,x_r\}.$ Let $I = \cap_{i=1}^{\alpha} P_i.$ Then for $s,n \in \NN$,
	\begin{align} \label{cor1:eq1}
		\ell\left(\frac{S}{I+\mm^{[s]}\mm^n}\right) 
		&= \sum_{i=1}^{\alpha} \ell\left(\frac{S}{P_i+\mm^{[s]}\mm^n}\right) - \sum_{1 \leq i<j \leq \alpha} \ell\left(\frac{S}{P_i+P_j+\mm^{[s]}\mm^n}\right)+ \cdots \nonumber\\
		&\hspace{7cm}+(-1)^{\alpha-1} \ell\left(\frac{S}{\sum_{i=1}^{\alpha} P_i+\mm^{[s]}\mm^n}\right).
	\end{align}
	In particular, $\ell\left(\frac{S}{I+\mm^{[s]}\mm^n}\right)$ is a piecewise polynomial in $s$ and $n.$ 
\end{Corollary}

\begin{proof}
	Since $S/\mm^{[s]}\mm^n$ is a standard graded Artinian ring, using Theorem \ref{thm1:alpha2} it follows that
	\begin{align} \label{eq1}
	 H\left(\frac{S}{I+\mm^{[s]}\mm^n}, \lambda \right) 
	 &= \sum_{i=1}^{\alpha} H\left(\frac{S}{P_i+\mm^{[s]}\mm^n}, \lambda \right) - \sum_{\substack{i,j=1\\i<j}}^{\alpha} H\left(\frac{S}{P_i+P_j+\mm^{[s]}\mm^n}, \lambda \right)+ \cdots \nonumber\\ 
	 &\hspace{6cm}+(-1)^{\alpha-1} H\left(\frac{S}{\sum_{i=1}^{\alpha} P_i+\mm^{[s]}\mm^n}, \lambda \right).
	\end{align}
	The modules involved on the right side of (\ref{eq1}) are finite length $S$-modules. Put $\lambda=1$ in (\ref{eq1}) to get (\ref{cor1:eq1}).
%	\begin{align*}
%	\ell\left(\frac{S}{I+\mm^{[s]}\mm^n}\right) 
%	&= \sum_{i=1}^{\alpha} \ell\left(\frac{S}{P_i+\mm^{[s]}\mm^n}\right) - \sum_{\substack{i,j=1\\i<j}}^{\alpha} \ell\left(\frac{S}{P_i+P_j+\mm^{[s]}\mm^n}\right)+ \cdots \\
%	&\hspace{7cm}+(-1)^{\alpha-1} \ell\left(\frac{S}{\sum_{i=1}^{\alpha} P_i+\mm^{[s]}\mm^n}\right).
%	\end{align*}
	Observe that $S/(P_{i_1}+\cdots+P_{i_j})$, for $i_1,\ldots,i_j \in \{1,\ldots,\alpha\}$, is isomorphic to a polynomial ring. Since image of $\mm$ in $S/(P_{i_1}+\cdots+P_{i_j})$ is a parameter ideal for all  $i_1,\ldots,i_j \in \{1,\ldots,\alpha\}$, using Corollary \ref{corF(n)} we obtain the required result.
\end{proof}

We are now ready to prove the main result of the section. We first consider the general case.

\subsection{The generalized Hilbert-Kunz function of $(\nn,\nn t)$}

\begin{Theorem} \label{NCM}
	Let $S = k[x_1,\ldots,x_r]$ be a polynomial ring in $r$ variables over a field $k$ and let $\mm$ be the maximal homogeneous ideal of $S.$ Let $P_1,\ldots,P_\alpha$, for $\alpha \geq 2$, be distinct $S$-ideals generated by subsets of $\{ x_1,\ldots,x_r \}.$ Let $I = \cap_{i=1}^\alpha P_i$ and $R=S/I.$ Suppose $\nn = \mm/I$ denotes the maximal homogeneous ideal of $R$ and $\dim(R)=d.$ Set $\delta = \max\{ |a_i(R)| \colon a_i(R) \neq - \infty \}.$
	Then for  $s > \delta$,
	\begin{align*}
	\ell\left(\frac{\R(\nn)}{(\nn,\nn t)^{[s]}}\right)
	\end{align*}
	is a polynomial in $s.$
\end{Theorem}

\begin{proof}
	Since $R$ is a standard graded ring, it follows that $R \simeq G(\nn).$ Let $s > \delta.$ Using Theorem \ref{Hoa}, it follows that
	\begin{align*}
	r(\nn^s) = 
	\begin{cases}
	d-1 & \text{ if } a_d(R)<0, \\
	d & \text{ if } a_d(R)=0.
	\end{cases}
	\end{align*}
	In other words, $r(\nn^s)=d-j$, where $j$ is either $0$ or $1$ as per the above observation. As $\nn^{[s]}$ is a minimal reduction of $\nn^s$, we get, $\nn^{[s]}\nn^{(d-j)s} = \nn^{(d-j+1)s}.$ In other words, $\nn^{[s]} \nn^{n-s} = \nn^n$, for all $n \geq (d-j+1)s.$ This implies that
	\begin{align*}
	(\nn,\nn t)^{[s]} 
	= (\nn^{[s]}, \nn^{[s]}t^s) 
	&= \left(\bigoplus_{n=0}^{s-1} \nn^{[s]}\nn^nt^n \right) + \left(\bigoplus_{n \geq s}  \nn^{[s]}\nn^{n-s} t^n \right) \\
	&= \left(\bigoplus_{n=0}^{s-1} \nn^{[s]}\nn^nt^n \right) + \left(\bigoplus_{n=s}^{(d-j+1)s-1} \nn^{[s]}\nn^{n-s} t^n \right) + \left(\bigoplus_{n \geq (d-j+1)s} \nn^nt^n \right). 
	\end{align*}
	Therefore, for  $s > \delta$,
	\begin{align*}
	\ell\left(\frac{\mathcal{R}(\nn)}{(\nn,\nn t)^{[s]}}\right)
	&= \sum_{n=0}^{s-1} \ell\left(\frac{\nn^n}{\nn^{[s]}\nn^n}\right) 
	+ \sum_{n=s}^{(d-j+1)s-1} \ell\left(\frac{\nn^n}{\nn^{[s]}\nn^{n-s}} \right) \\
	&= \sum_{n=0}^{s-1} \ell\left(\frac{R}{\nn^{[s]}\nn^n}\right) 
	+ \sum_{n=s}^{(d-j+1)s-1} \ell\left(\frac{R}{\nn^{[s]}\nn^{n-s}} \right) 
	- \sum_{n=0}^{(d-j+1)s-1} \ell\left(\frac{R}{\nn^n} \right) \\
	&= \sum_{n=1}^{s-1} \ell\left(\frac{S}{I+\mm^{[s]}\mm^n}\right) 
	+ \sum_{n=s+1}^{(d-j+1)s-1} \ell\left(\frac{S}{I+\mm^{[s]}\mm^{n-s}} \right) 
	- \sum_{n=1}^{(d-j+1)s-1} \ell\left(\frac{R}{\nn^n} \right) + 2 \ \ell\left(\frac{R}{\nn^{[s]}}\right) \\
	&= 2 \ \sum_{n=1}^{s-1} \ell\left(\frac{S}{I+\mm^{[s]}\mm^n}\right) 
	+ \sum_{n=s}^{(d-j)s-1} \ell\left(\frac{S}{I+\mm^{[s]}\mm^n} \right) 
	- \sum_{n=1}^{(d-j+1)s-1} \ell\left(\frac{R}{\nn^n} \right) + 2 \ \ell\left(\frac{R}{\nn^{[s]}}\right).
	\end{align*}
	
	The result now follows from Corollary \ref{cor1:alpha2}, Theorem \ref{faceringpoly} and Theorem \ref{conca}.
\end{proof}

\subsection{The generalized Hilbert-Kunz function of $(\nn,\nn t)$ for Cohen-Macaulay $k[\Delta]$}

\begin{Theorem} \label{CM}
	Let $S = k[x_1,\ldots,x_r]$ be a polynomial ring in $r$ variables over a field $k$ and let $\mm$ be the maximal homogeneous ideal of $S.$ Let $P_1,\ldots,P_\alpha$, for $\alpha \geq 2$, be distinct $S$-ideals generated by subsets of $\{ x_1,\ldots,x_r \}.$ Let $I = \cap_{i=1}^\alpha P_i$ and $R=S/I.$ Suppose $\nn = \mm/I$ denotes the maximal homogeneous ideal of $R$ and $\dim(R)=d.$ Suppose that $R$ is Cohen-Macaulay. Then
	\begin{align*}
	\ell\left(\frac{\R(\nn)}{(\nn,\nn t)^{[s]}}\right)
	\end{align*}
	is given by a polynomial for $s\geq 1.$
\end{Theorem}

\begin{proof}
	Since $R$ is a standard graded ring, it follows that $R \simeq G(\nn).$ Let $h(\Delta)=(h_0,\ldots,h_d)$ denote the $h$-vector of $R.$ Note that $-d < n(\nn) \leq 0.$ If $n(\nn)=-d$, then $h_0=1$ and $h_i=0$ for all $i \neq 0$, implying that $0 = h_1 = r-d.$ It follows that $I$ is a height zero ideal, which is not true. Hence, $-d < n(\nn) \leq 0.$
	
	Suppose $n(\nn)=0.$ Using Theorem \ref{Marley}, it follows that $r(I^s)=d$, for all $s \geq 1.$ Using the same arguments as in the proof of Theorem \ref{NCM}, it follows that for $s \geq 1,$
	\begin{align*}
	\ell\left(\frac{\mathcal{R}(\nn)}{(\nn,\nn t)^{[s]}}\right)
	= 2 \ \sum_{n=1}^{s-1} \ell\left(\frac{S}{I+\mm^{[s]}\mm^n}\right) 
	+ \sum_{n=s}^{ds-1} \ell\left(\frac{S}{I+\mm^{[s]}\mm^n} \right) 
	- \sum_{n=1}^{(d+1)s-1} \ell\left(\frac{R}{\nn^n} \right) + 2 \ \ell\left(\frac{R}{\nn^{[s]}}\right).
	\end{align*}
	The result now follows from Corollary \ref{cor1:alpha2}, Theorem \ref{faceringpoly} and Theorem \ref{conca}.
	
	Now suppose that $n(\nn)<0.$ If $s < |n(\nn)|,$ write $|n(\nn)| = k_1s + k_2$, where $k_2 \in \{0,1,\ldots,s-1\}.$ Using Theorem \ref{Marley}, it follows that 
	\begin{align*}
	r(\nn^s) = 
	\begin{cases}
	d-k_1 & \text{ if } s < |n(\nn)|, \, k_2=0, \\
	d-k_1-1 & \text{ if } s < |n(\nn)|, \, k_2 \neq 0, \\
	d-1 & \text{ if } s \geq |n(\nn)|. 
	\end{cases}
	\end{align*}
	
	In other words, $r(\nn^s)=d-j,$ where $j \in \{1,k_1,k_1+1\}$ as per the above observation. Using the same arguments as in the proof of Theorem \ref{NCM}, we are done.
\end{proof}

\section{Examples}
 In this section, we illustrate the above results using some examples.

\begin{Example}{\rm
	Let $\Delta$ be the simplicial complex 
	\begin{center}
		\begin{tikzpicture}
		\draw (0,0) -- (1.5,0);
		\draw (2.5,0) -- (4,0);
		\filldraw (0,0) circle (2pt) node[anchor=north]{$x_1$};		
		\filldraw (1.5,0) circle (2pt) node[anchor=north]{$x_2$};
		\filldraw (2.5,0) circle (2pt) node[anchor=north]{$x_3$};
		\filldraw (4,0) circle (2pt) node[anchor=north]{$x_4$};
		\end{tikzpicture}
	\end{center}
	Then $R = k[x_1,x_2,x_3,x_4]/((x_1, x_2) \cap (x_3,x_4))$ is the face ring of $\Delta.$ Observe that $R$ is a $2$-dimensional ring with $f$-vector $f(\Delta) = (1,4,2)$ and $h$-vector $h(\Delta) = (1,2,-1).$ Set $S=k[x_1, x_2, x_3, x_4],$ $P_1 = (x_1,x_2)$, $P_2 = (x_3,x_4).$ Since $\depth(R)=1$, it follows that $a_0(R) = -\infty.$ In order to find $a_1(R)$ and $a_2(R)$, we consider the following short exact sequence. 
	\begin{align*}
		0 \longrightarrow \frac{S}{P_1 \cap P_2} \longrightarrow \frac{S}{P_1} \bigoplus \frac{S}{P_2} \longrightarrow \frac{S}{P_1+P_2} \longrightarrow 0
	\end{align*}
	Using the corresponding long exact sequence of local cohomology modules, it follows that 
	\[ \HH^1_{\nn}(R) \simeq \HH^0_{\mm}(S/(P_1+P_2)) \text{ and } \HH^2_{\nn}(R) \simeq \HH^2_{(x_3, x_4)}(k[x_3, x_4]) \oplus \HH^2_{(x_1,x_2)}(k[x_1,x_2]). \]
	This implies that $a_1(R) = 0$ and $a_2(R)=-2.$ Hence, $\delta = \max\{ |a_i(R)| \colon a_i(R) \neq - \infty \} = 2.$ Since $a_2(R)<0$, using Theorem \ref{NCM} it follows that for all $s > 2,$
	\begin{align*}
		\ell\left(\frac{\mathcal{R}(\nn)}{(\nn,\nn t)^{[s]}}\right)
		&= 2 \ \sum_{n=1}^{s-1} \ell\left(\frac{S}{I+\mm^{[s]}\mm^n}\right) 
		- \sum_{n=1}^{2s-1} \ell\left(\frac{R}{\nn^n} \right) + 2 \ \ell\left(\frac{R}{\nn^{[s]}}\right).
	\end{align*}
	From Corollary \ref{cor1:alpha2}, Theorem \ref{faceringpoly} and Theorem \ref{conca}, we obtain
	\begin{align*}
		\ell\left(\frac{\mathcal{R}(\nn)}{(\nn,\nn t)^{[s]}}\right)
		&= 2 \sum_{n=1}^{s-1} \left[ \ell\left(\frac{S}{P_1+\mm^{[s]}\mm^n}\right) + \ell\left(\frac{S}{P_2+\mm^{[s]}\mm^n}\right) - \ell\left(\frac{S}{P_1+P_2+\mm^{[s]}\mm^n}\right) \right] \\
		&- \sum_{i=1}^{2s-1} \left[ \sum_{i=0}^{2} (-1)^i \frac{\h^{(i)}(1)}{i!} \binom{n+1-i}{2-i} \right] + 2 \sum_{i=0}^{2} f_{i-1} (s-1)^i.
	\end{align*}
	Substituting the values and using Corollary \ref{corF(n)}, we get
	\begin{align*}
		\ell\left(\frac{\mathcal{R}(\nn)}{(\nn,\nn t)^{[s]}}\right)
		= 2 \sum_{n=1}^{s-1} \bigg[ 2(s^2+n^2+n) -1 \bigg] - \sum_{n=1}^{2s-1} \bigg[ 2\binom{n+1}{2} -1 \bigg] + 2 \bigg[ 1 + 4(s-1) + 2(s-1)^2 \bigg].
	\end{align*}
	Simplifying the above expression, we obtain that for all $s > 2$,
	\begin{align*}
		\ell\left(\frac{\mathcal{R}(\nn)}{(\nn,\nn t)^{[s]}}\right)
		&= \frac{8}{3} s^3 - \frac{2}{3} s - 1 \\
		&= 16 \binom{s+2}{3} - 16 \binom{s+1}{2} + 2s - 1.
	\end{align*}
}\end{Example}

\begin{Example}{\rm
	Let $\Delta$ be the simplicial complex
	
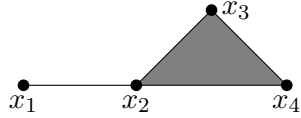
\begin{center}
		\begin{tikzpicture}
		\draw (0,0) -- (1.5,0);
		\filldraw[fill=gray, draw=black] (1.5,0) -- (2.5,1) -- (3.5,0) -- cycle;
		\filldraw (0,0) circle (2pt) node[anchor=north]{$x_1$};		
		\filldraw (1.5,0) circle (2pt) node[anchor=north]{$x_2$};
		\filldraw (2.5,1) circle (2pt) node[anchor=west]{$x_3$};
		\filldraw (3.5,0) circle (2pt) node[anchor=north]{$x_4$};
		\end{tikzpicture}
	\end{center}
	Then $R = k[x_1,x_2,x_3,x_4]/((x_1) \cap (x_3,x_4))$ is the face ring of $\Delta.$ Observe that $R$ is a $3$-dimensional ring with $f$-vector $f(\Delta) = (1,4,4,1)$ and $h$-vector $h(\Delta) = (1,1,-1,0).$ Set $S=k[x_1, x_2, x_3, x_4],$ $P_1 = (x_3,x_4)$, $P_2 = (x_1).$ Since $\depth(R)=2$, it follows that $a_0(R) = a_1(R) = -\infty.$ In order to find $a_2(R)$ and $a_3(R)$, we consider the following short exact sequence. 
	\begin{align*}
		0 \longrightarrow \frac{S}{P_1 \cap P_2} \longrightarrow \frac{S}{P_1} \bigoplus \frac{S}{P_2} \longrightarrow \frac{S}{P_1+P_2} \longrightarrow 0
	\end{align*}
	Using the corresponding long exact sequence of local cohomology modules, we get 
	\[ \HH^3_{\nn}(R) \simeq \HH^3_{(x_2, x_3, x_4)}(k[x_2, x_3, x_4]) \text{ and } 0 \rightarrow \HH^1_{(x_2)}(k[x_2]) \rightarrow \HH^2_{\nn}(R) \rightarrow \HH^2_{(x_1,x_2)}(k[x_1, x_2]) \rightarrow 0. \]
	This implies that $a_2(R) = -1$ and $a_3(R)=-3.$ Hence, $\delta = \max\{ |a_i(R)| \colon a_i(R) \neq - \infty \} = 3.$ Since $a_3(R)<0$, using Theorem \ref{NCM} it follows that for all $s > 3,$
	\begin{align*}
	\ell\left(\frac{\mathcal{R}(\nn)}{(\nn,\nn t)^{[s]}}\right)
	&= 2 \ \sum_{n=1}^{s-1} \ell\left(\frac{S}{I+\mm^{[s]}\mm^n}\right) 
	+ \sum_{n=s}^{2s-1} \ell\left(\frac{S}{I+\mm^{[s]}\mm^n} \right) 
	- \sum_{n=1}^{3s-1} \ell\left(\frac{R}{\nn^n} \right) + 2 \ \ell\left(\frac{R}{\nn^{[s]}}\right).
	\end{align*}
	From Corollary \ref{cor1:alpha2}, Theorem \ref{faceringpoly} and Theorem \ref{conca}, we obtain
	\begin{align*}
	\ell\left(\frac{\mathcal{R}(\nn)}{(\nn,\nn t)^{[s]}}\right)
	&= 2 \sum_{n=1}^{s-1} \left[ \ell\left(\frac{S}{P_1+\mm^{[s]}\mm^n}\right) + \ell\left(\frac{S}{P_2+\mm^{[s]}\mm^n}\right) - \ell\left(\frac{S}{P_1+P_2+\mm^{[s]}\mm^n}\right) \right] \\
	&+\sum_{n=s}^{2s-1} \left[ \ell\left(\frac{S}{P_1+\mm^{[s]}\mm^n}\right) + \ell\left(\frac{S}{P_2+\mm^{[s]}\mm^n}\right) - \ell\left(\frac{S}{P_1+P_2+\mm^{[s]}\mm^n}\right) \right] \\
	&- \sum_{i=1}^{3s-1} \left[ \sum_{i=0}^{3} (-1)^i \frac{\h^{(i)}(1)}{i!} \binom{n+d-i-1}{d-i} \right] + 2 \sum_{i=0}^{3} f_{i-1} (s-1)^i.
	\end{align*}
	Substituting the values and using Corollary \ref{corF(n)}, we get
	\begin{align*}
	\ell\left(\frac{\mathcal{R}(\nn)}{(\nn,\nn t)^{[s]}}\right)
	&= 2 \sum_{n=1}^{s-1} \bigg[ (s^2+n^2+n) + s^3 + 3 \binom{n+2}{3} -(s+n) \bigg] \\
	&+ \sum_{n=s}^{2s-1} \bigg[ \binom{n+s+1}{2} + s^3 + 3\binom{n+2}{3} -3 \binom{n-s+2}{3} -(s+n) \bigg] \\
	&- \sum_{n=1}^{3s-1} \bigg[ \binom{n+2}{3} + \binom{n+1}{2} -n \bigg] + 2 \bigg[ 1 + 4(s-1) + 4(s-1)^2 + (s-1)^3 \bigg].
	\end{align*}
	Simplifying the above expression, we obtain that for all $s > 3$,
	\begin{align*}
	\ell\left(\frac{\mathcal{R}(\nn)}{(\nn,\nn t)^{[s]}}\right)
	&= \frac{13}{8} s^4 + \frac{13}{12} s^3 - \frac{9}{8} s^2 - \frac{7}{12} s \\
	&= 39 \binom{s+3}{4} - 52 \binom{s+2}{3} + 14 \binom{s+1}{2}.
	\end{align*}
}\end{Example}

\begin{Example}{\rm
	Let $\Delta$ be the simplicial complex 
	
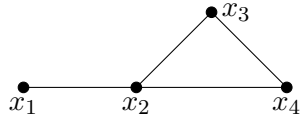
\begin{center}
		\begin{tikzpicture}
		\draw (0,0) -- (1.5,0);
		\draw (1.5,0) -- (2.5,1) -- (3.5,0) -- cycle;
		\filldraw (0,0) circle (2pt) node[anchor=north]{$x_1$};		
		\filldraw (1.5,0) circle (2pt) node[anchor=north]{$x_2$};
		\filldraw (2.5,1) circle (2pt) node[anchor=west]{$x_3$};
		\filldraw (3.5,0) circle (2pt) node[anchor=north]{$x_4$};
		\end{tikzpicture}
	\end{center}
	Then $R = k[x_1,x_2,x_3,x_4]/((x_3,x_4) \cap (x_1,x_3) \cap (x_1,x_4) \cap (x_1,x_2))$ is the face ring of $\Delta.$ Observe that $R$ is a $2$-dimensional Cohen-Macaulay ring with $f$-vector $f(\Delta) = (1,4,4)$ and $h$-vector $h(\Delta) = (1,2,1).$ This implies that $n(\nn)=0.$ Using Theorem \ref{CM}, it follows that for $s \geq 1,$
	\begin{align*}
		\ell\left(\frac{\mathcal{R}(\nn)}{(\nn,\nn t)^{[s]}}\right)
		&= 2 \ \sum_{n=1}^{s-1} \ell\left(\frac{S}{I+\mm^{[s]}\mm^n}\right) 
		+ \sum_{n=s}^{2s-1} \ell\left(\frac{S}{I+\mm^{[s]}\mm^n} \right) 
		- \sum_{n=1}^{3s-1} \ell\left(\frac{R}{\nn^n} \right) + 2 \ \ell\left(\frac{R}{\nn^{[s]}}\right).
	\end{align*}
	Substituting, we get
	\begin{align*}
		\ell\left(\frac{\mathcal{R}(\nn)}{(\nn,\nn t)^{[s]}}\right)
		&= 2 \sum_{n=1}^{s-1} \bigg[ 4(s^2+n^2+n) - 4(s+n) + 1 \bigg]
		+ \sum_{n=s}^{2s-1} \bigg[ 4\binom{n+s+1}{2} - 4(s+n) + 1 \bigg] \\
		&- \sum_{n=1}^{3s-1} \bigg[ 4\binom{n+1}{2} - 4n + 1 \bigg] + 2 \bigg[ 1 + 4(s-1) + 4(s-1)^2 \bigg].
	\end{align*}
	Simplifying the above expression, we obtain that for all $s \geq 1$,
	\begin{align*}
		\ell\left(\frac{\mathcal{R}(\nn)}{(\nn,\nn t)^{[s]}}\right)
		&= \frac{16}{3} s^3 - 4 s^2 - \frac{4}{3} s + 1\\
		&= 32 \binom{s+2}{3} - 40 \binom{s+1}{2} + 8s + 1.
	\end{align*}
}\end{Example}

\begin{Example}{\rm
	Let $\Delta$ be a $1$-dimensional simplicial complex on $r$ vertices, for some $r \geq 3:$
	\medskip
	\begin{center}
		\begin{tikzpicture}
		\draw (0,0) -- (2.5,0);
		\filldraw (0,0) circle (2pt) node[anchor=north]{$x_1$};
		\filldraw (1,0) circle (2pt) node[anchor=north]{$x_2$};
		\filldraw (2,0) circle (2pt) node[anchor=north]{$x_3$};
		\draw [dashed] (3,0) -- (4,0);
		\draw (4.5,0) -- (6,0);
		\filldraw (5,0) circle (2pt) node[anchor=north]{$x_{r-1}$};
		\filldraw (6,0) circle (2pt) node[anchor=north]{$x_r$};
		\end{tikzpicture}
	\end{center}
	
	For $i=1,\ldots,r-1$, set $P_i = \big( \{x_1,\ldots,x_r\} \setminus \{x_i, x_{i+1}\} \big).$ Then $R=k[x_1,\ldots,x_r]/\cap_{i=1}^{r-1} P_i$ is the face ring of $\Delta.$ It is a two-dimensional Cohen-Macaulay ring with $f$-vector $f(\Delta) = (1,r,r-1)$ and $h$-vector $h(\Delta) = (1,r-2,0).$ Since the $a$-invariant $a_{2}(R)=-1$, using Theorem \ref{CM}, it follows that for $s \geq 1,$
	\begin{align*}
	&\ell\left(\frac{\mathcal{R}(\nn)}{(\nn,\nn t)^{[s]}}\right) \\
	&= 2 \sum_{n=1}^{s-1} \ell\left(\frac{S}{I+\mm^{[s]}\mm^n}\right) - \sum_{n=1}^{2s-1} \ell\left(\frac{R}{\nn^n}\right) + 2 \ \ell\left(\frac{R}{\nn^{[s]}}\right) \\
	&= 2 \sum_{n=1}^{s-1} \left[ \sum_{i=1}^{r-1} \ell\left(\frac{S}{P_i+\mm^{[s]}\mm^n}\right) - \sum_{\substack{i,j=1\\i<j}}^{r-1} \ell\left(\frac{S}{P_i+P_j+\mm^{[s]}\mm^n}\right)+ \cdots+(-1)^{r-2} \ell\left(\frac{S}{\sum_{i=1}^{r-1} P_i+\mm^{[s]}\mm^n}\right) \right] \\
	&- \sum_{n=1}^{2s-1} \sum_{i=0}^{2} (-1)^i \frac{\h^{(i)}(1)}{i!} \binom{n+1-i}{2-i} + 2 \sum_{i=0}^{2} f_{i-1} (s-1)^i.
	\end{align*}
	Observe that in this case, using Corollary \ref{corF(n)}, it follows that $\ell(S/(P_i+\mm^{[s]}\mm^n)) = s^2+n^2+n$, for all $1 \leq n \leq s-1$ and for all $i=1,\ldots,r-1.$ For $1 \leq i < j \leq r-1$, if $\{x_i,x_{i+1}\} \cap \{x_{j}, x_{j+1}\} \neq \emptyset$, then $S/(P_i+P_j) \simeq k{[x]}$ and there are $r-2$ such instances. Otherwise, $S/(P_i+P_j) \simeq k.$ It is also easy to observe that $S/(P_{i_1}+\cdots+P_{i_u}) \simeq k,$ for all $u \geq 3$ and $i_1,\ldots,i_u \in \{1,\ldots,r-1 \}.$ Therefore,
	\begin{align*}
	&\ell\left(\frac{\mathcal{R}(\nn)}{(\nn,\nn t)^{[s]}}\right) \\
	&= 2 \sum_{n=1}^{s-1} \left[ (r-1)(s^2+n^2+n) - \Big[(r-2)(s+n) + \binom{r-1}{2} - (r-2) \Big] + \binom{r-1}{3} + \cdots + (-1)^{r-2} \right] \\
	&- \sum_{n=1}^{2s-1} \left[ (r-1)\binom{n+1}{2} - (r-2)n \right] + 2 \left[ 1+r(s-1) + (r-1)(s-1)^2 \right].
	\end{align*}
	Since, $\sum_{i=2}^{r-1} (-1)^i \binom{r-1}{i} = r-2,$ simplifying the above expression we get
	\begin{align*}
	\ell\left(\frac{\mathcal{R}(\nn)}{(\nn,\nn t)^{[s]}}\right)
	&= \frac{4}{3} (r-1)s^3 - (r-2) s^2 - \frac{(r-1)}{3} s \\
	&= 8(r-1) \binom{s+2}{3} - 2(5r-6)\binom{s+1}{2} + (2r-3)s.
	\end{align*}
}\end{Example} 

  We need some terminologies for the next example.  
 \begin{Definition}
Let $G$ be a finite simple graph with vertices $V=V(G)=\{x_1, \ldots, x_n\}$ and the edges $E=E(G)$. The edge ideal of $I(G)$ of $G$ is defined to be the ideal in $K[x_1,\ldots,x_n]$ generated by the square free quadratic monomials representing the edges of $G$, i.e., 
\[ I(G) = \langle x_ix_j \mid x_ix_j \in E \rangle. \]
\end{Definition}

A vertex cover of a graph is a set of vertices such that every edge has at least one vertex belonging to that set. A minimal vertex cover is a vertex cover such that none of its subsets is a vertex cover. For any graph $G$ with the set of all minimal vertex covers $C$, the edge ideal $I(G)$ has the primary decomposition: 
\[ I(G)= \bigcap_{\{x_{i_1}, \ldots, x_{i_u}\}\in C} (x_{i_1}, \ldots, x_{i_u}). \]

For example, when $G$ is a five cycle, the primary decomposition of the edge ideal $I(G)$ is 
\[ I(G) = (x_1x_2, x_2x_3, x_3x_4, x_4x_5, x_5x_1) = (x_1,x_2,x_4) \cap (x_1,x_3,x_5) \cap (x_1,x_3,x_4) \cap (x_2,x_3,x_5) \cap (x_2,x_4,x_5). \]

\begin{Example}[\bf Complete Bipartite Graphs]{\rm

A complete bipartite graph $K_{\alpha,\beta}$ is a graph whose set of  vertices is decomposed into two disjoint sets such that no two  vertices within the same set are adjacent and that every pair of  vertices in the two sets are adjacent.

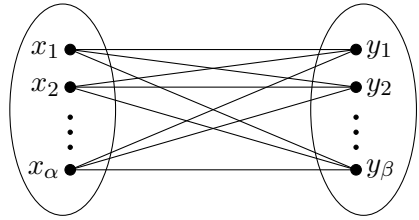
\begin{figure}[h]
	\begin{tikzpicture}
	\draw (0,0) ellipse (20pt and 40pt);
	\draw (4,0) ellipse (20pt and 40pt);	
	\filldraw (0.1,.8) circle (2pt) node[anchor=east]{$x_1$};
	\filldraw (0.1,.3) circle (2pt) node[anchor=east]{$x_2$};	
	\filldraw (0.1,-.1) circle (.8pt);	
	\filldraw (0.1,-.3) circle (.8pt);
	\filldraw (0.1,-.5) circle (.8pt);
	\filldraw (0.1,-.8) circle (2pt) node[anchor=east]{$x_\alpha$};	
	\filldraw (3.9,.8) circle (2pt) node[anchor=west]{$y_1$};
	\filldraw (3.9,.3) circle (2pt) node[anchor=west]{$y_2$};	
	\filldraw (3.9,-.1) circle (.8pt);	
	\filldraw (3.9,-.3) circle (.8pt);
	\filldraw (3.9,-.5) circle (.8pt);
	\filldraw (3.9,-.8) circle (2pt) node[anchor=west]{$y_\beta$};	
	\draw (.1,.8) -- (3.9,.8);
	\draw (.1,.8) -- (3.9,.3);
	\draw (.1,.8) -- (3.9,-.8);
	\draw (.1,.3) -- (3.9,.8);
	\draw (.1,.3) -- (3.9,.3);
	\draw (.1,.3) -- (3.9,-.8);
	\draw (.1,-.8) -- (3.9,.8);
	\draw (.1,-.8) -- (3.9,.3);
	\draw (.1,-.8) -- (3.9,-.8);
	\end{tikzpicture}
	\caption{$K_{\alpha,\beta}$}
\end{figure}

Let $S=k[x_1, \ldots, x_\alpha, y_1, \ldots, y_\beta],$  where $3 \leq \alpha \leq \beta.$ The edge ideal of $K_{\alpha,\beta}$ is the ideal $I = \big(x_i y_j \mid 1 \leq i \leq \alpha, 1\leq j \leq \beta \big).$ Observe that $R=S/I$ is a $\beta$-dimensional ring. Let  $P_1 = (x_1, \ldots, x_\alpha)$, $P_2 = (y_1, \ldots, y_\beta).$ Then $I=P_1 \cap P_2$. Note that $I$ is the Stanley-Reisner ideal of the union of an $\alpha$-simplex and a $\beta$-simplex. 

In order to find the $a$-invariants  we consider the following short exact sequence. 
\begin{align*}
	0 \longrightarrow \frac{S}{P_1 \cap P_2} \longrightarrow \frac{S}{P_1} \bigoplus \frac{S}{P_2} \longrightarrow \frac{S}{P_1+P_2} \longrightarrow 0
\end{align*}
Using the corresponding long exact sequence of local cohomology modules, it follows that 
\[ \HH^1_{\nn}(R) \simeq \HH^0_{\mm}\left(\frac{S}{P_1+P_2}\right), \, \HH^\alpha_{\nn}(R) \simeq \HH^\alpha_{(x_1,\ldots x_\alpha)}(k[x_1,\ldots, x_\alpha]) \text{ and } \HH^\beta_{\nn}(R) \simeq \HH^\beta_{(y_1,\ldots, y_\beta)}(k[y_1,\ldots, y_\beta]). \]
Therefore, $a_1(R) = 0$, $a_\alpha(R)=-\alpha$ and $a_\beta(R)=-\beta.$ Hence, $\delta = \max\{ |a_i(R)| \colon a_i(R) \neq - \infty \} = \beta.$ Since $a_\beta(R)<0$, using Theorem \ref{NCM} it follows that for all $s > \beta,$
\begin{align*}
	\ell\left(\frac{\mathcal{R}(\nn)}{(\nn,\nn t)^{[s]}}\right)
	&= 2 \ \sum_{n=1}^{s-1} \ell\left(\frac{S}{I+\mm^{[s]}\mm^n}\right) 
	+ \sum_{n=s}^{(\beta-1)s-1} \ell\left(\frac{S}{I+\mm^{[s]}\mm^n}\right) 
	- \sum_{n=1}^{\beta s-1} \ell\left(\frac{R}{\nn^n} \right) + 2 \ \ell\left(\frac{R}{\nn^{[s]}}\right).
\end{align*}
From Corollary \ref{cor1:alpha2}, Theorem \ref{faceringpoly} and Theorem \ref{conca}, we obtain
\begin{align*}
	\ell\left(\frac{\mathcal{R}(\nn)}{(\nn,\nn t)^{[s]}}\right)
	&= 2 \sum_{n=1}^{s-1} \left[ \ell\left(\frac{S}{P_1+\mm^{[s]}\mm^n}\right) + \ell\left(\frac{S}{P_2+\mm^{[s]}\mm^n}\right) - \ell\left(\frac{S}{P_1+P_2+\mm^{[s]}\mm^n}\right) \right] \\
	&+ \sum_{n=s}^{(\beta - 1) s-1} \left[ \ell\left(\frac{S}{P_1+\mm^{[s]}\mm^n}\right) + \ell\left(\frac{S}{P_2+\mm^{[s]}\mm^n}\right) - \ell\left(\frac{S}{P_1+P_2+\mm^{[s]}\mm^n}\right) \right] \\
	&- \sum_{n=1}^{\beta s-1} \left[ \sum_{i=0}^{\beta} (-1)^i \frac{\h^{(i)}(1)}{i!} \binom{n+\beta-i-1}{\beta-i} \right] + 2 \sum_{i=0}^{\beta} f_{i-1} (s-1)^i.
\end{align*}

As the $f$-vector is $f(\Delta) = \left(1,\alpha+\beta,\binom{\alpha}{2}+\binom{\beta}{2},\ldots,\binom{\alpha}{\alpha}+\binom{\beta}{\alpha}, \binom{\beta}{\alpha+1},\ldots,\binom{\beta}{\beta}\right)$ and the $h$-vector can be computed using \cite[Lemma 5.1.8]{brunsHerzog}, substituting the values and using Corollary \ref{corF(n)}, it follows that for all $s > \beta,$
\begin{align*}
	\ell\left(\frac{\mathcal{R}(\nn)}{(\nn,\nn t)^{[s]}}\right)
	&= 2 \sum_{n=1}^{s-1} \bigg[ s^\beta + \beta \binom{n+\beta-1}{\beta} + s^\alpha + \alpha \binom{n+\alpha-1}{\alpha} - 1 \bigg] \\
	&+  \sum_{n=s}^{(\beta-1)s-1} \bigg[ s^\beta + \sum_{i=1}^{\beta-1} (-1)^{i+1} \binom{\beta}{i} \binom{n-(i-1)s+\beta-1}{\beta} - 1 \bigg] \\
	&+ \sum_{n=s}^{(\alpha-1)s-1} \bigg[ s^\alpha + \sum_{i=1}^{\alpha-1} (-1)^{i+1} \binom{\alpha}{i} \binom{n-(i-1)s+\alpha-1}{\alpha} \bigg] + \sum_{n=(\alpha-1)s}^{(\beta-1)s-1} \binom{n+s+\alpha-1}{\alpha} \\
	&- \sum_{n=1}^{\beta s-1} \bigg[ \sum_{i=0}^{\beta} (-1)^i \frac{\h^{(i)}(1)}{i!} \binom{n+\beta-i-1}{\beta-i} \bigg] + 2 \bigg[ 1+ \sum_{i=1}^{\beta} \binom{\beta}{i} (s-1)^i + \sum_{i=1}^{\alpha} \binom{\alpha}{i} (s-1)^i  \bigg].
\end{align*}

In particular, when $\alpha=3$ and $\beta=4,$ we obtain that for all $s > 4,$
\begin{align*}
	\ell\left(\frac{\mathcal{R}(\nn)}{(\nn,\nn t)^{[s]}}\right)
	&= \frac{61}{30} s^5 + \frac{19}{24} s^4 - \frac{1}{12} s^3 - \frac{7}{24} s^2 - \frac{9}{20} s -1.
\end{align*}

}\end{Example}

Sometimes, certain invariants of the Stanley-Reisner ring may depend on the characteristic of the ring. Triangulation of the real projective plane is one such example where the Cohen-Macaulay property of the ring is characteristic dependent. We prove that in this example, the Hilbert-Kunz function is characteristic independent. 

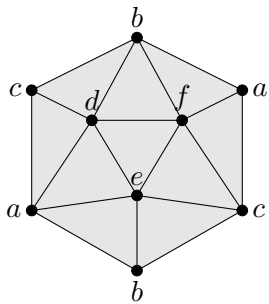
\begin{Example}[\bf Triangulation of real projective plane]{\rm
Let $\Delta$ be the triangulation of the real projective plane.
%\begin{figure}[h]
%	\begin{center}
%		\includegraphics[width=1.5in,height=1.5in]{Triangulation.jpg}
%	\end{center}
%\end{figure}
\begin{figure}[h]
	\begin{tikzpicture}	
	\draw[fill=gray!20!white, draw=black] (0,0) -- (-1.4,.8) -- (-1.4,2.4) -- (0,3.1) -- (1.4,2.4) -- (1.4,0.8) -- cycle;
	\filldraw (0,0) circle (2pt) node[anchor=north]{$b$};
	\filldraw (-1.4,0.8) circle (2pt) node[anchor=east]{$a$};	
	\filldraw (0,1) circle (2pt) node[anchor=south]{$e$};	
	\filldraw (1.4,0.8) circle (2pt) node[anchor=west]{$c$};
	\filldraw (-.6,2) circle (2pt) node[anchor=south]{$d$};
	\filldraw (.6,2) circle (2pt) node[anchor=south]{$f$};	
	\filldraw (-1.4,2.4) circle (2pt) node[anchor=east]{$c$};
	\filldraw (1.4,2.4) circle (2pt) node[anchor=west]{$a$};	
	\filldraw (0,3.1) circle (2pt) node[anchor=south]{$b$};	
	\draw (0,1) -- (-0.6,2) -- (.6,2) -- cycle;
	\draw (0,1) -- (-1.4,.8);
	\draw (0,1) -- (0,0);
	\draw (0,1) -- (1.4,.8);
	\draw (-.6,2) -- (-1.4,.8);
	\draw (-.6,2) -- (-1.4,2.4);
	\draw (-.6,2) -- (0,3.1);
	\draw (.6,2) -- (1.4,.8);
	\draw (.6,2) -- (1.4,2.4);
	\draw (.6,2) -- (0,3.1);
	\end{tikzpicture}
\end{figure}

Let $k$ be a field and $R$ be the corresponding Stanley-Reisner ring of $\Delta.$ It is known that $R$ is Cohen-Macaulay if and only if $\chara k \neq 2.$ The $f$-vector of $R$ is $f(\Delta) = (1,6,15,10)$ and $h$-vector of $R$ is $h(\Delta)=(1,3,6,0).$ Let $\chara k \neq 2.$ Then $R$ is Cohen-Macaulay and $n(\nn)=-1.$ Using \texttt{Macaulay2} code, we obtain that for $s \geq 1$,
\begin{align} \label{triang}
	\ell\left(\frac{\mathcal{R}(\nn)}{(\nn,\nn t)^{[s]}}\right)
	= 390 \binom{s+3}{4} - 720 \binom{s+2}{3} + 372 \binom{s+1}{2} - 41s.
\end{align}

\vskip2mm
We save the code in a file named as \texttt{HKPolySC.m2} and make the following session in Macaulay2.

{\small\begin{verbatim}
	
i1 : S = QQ[a..f];
i2 : I = ideal"abe, ade, acd, bcd, bdf, abf, acf, cef, bce, def";
i3 : needsPackage"Depth"
i4 : needsPackage"SimplicialComplexes"
i5 : needsPackage"SimplicialDecomposability"
i6 : load"HKPolySC.m2"
i7 : HKPolySC(I)
The postulation number is: -1
Enter a number bigger than or equal to the absolute value of the postulation number: 2
The value of the Hilbert-Kunz function at the point 2 is: 104
Do you wish to enter one more point? (true/false): true
Enter a number bigger than or equal to the absolute value of the postulation number: 3
The value of the Hilbert-Kunz function at the point 3 is: 759
Do you wish to enter one more point? (true/false): true
Enter a number bigger than or equal to the absolute value of the postulation number: 4
The value of the Hilbert-Kunz function at the point 4 is: 2806
Do you wish to enter one more point? (true/false): true
Enter a number bigger than or equal to the absolute value of the postulation number: 5
The value of the Hilbert-Kunz function at the point 5 is: 7475
Do you wish to enter one more point? (true/false): true
Enter a number bigger than or equal to the absolute value of the postulation number: 6
The value of the Hilbert-Kunz function at the point 6 is: 16386
Do you wish to enter one more point? (true/false): false

\end{verbatim}}

One may check that if $\chara k =2$, then the $a$-invariant of $R$ is negative and $\depth (R)=2.$ Using Theorem \ref{NCM} it follows that $\ell(\R(\nn)/(\nn,\nn t)^{[s]})$ has the same formula as in (\ref{triang}) for $s > \delta$, where $\delta = \max\{ |a_2(R)|, |a_3(R)| \}.$ This proves that the Hilbert-Kunz function is characteristic independent in this example. 

%\textcolor{red}{\bf How to calculate the $a$-invariant in non-CM case?} Using %Macaulay2, I could check that in char 2, ring has depth 2 and dimension 3. So we %need to only calculate $a_2(R)$ and $a_3(R)$. If $a_3(R)<0$ in char 2, then we %obtain the same polynomial as above. Macaulay2 indicates that this is indeed true. %This will say that at least in this example, the HK polynomial is characteristic %independent.
}\end{Example}

\section{Macaulay2 code for Cohen-Macaulay Stanley-Reisner rings}

In this section we present a Macaulay2 code which uses the idea of Theorem \ref{CM} to calculate the value of the generalized Hilbert-Kunz function at a point. The code requires Macaulay2 packages SimplicialComplexes, SimplicialDecomposability and Depth. The code accepts the Stanley-Reisner ideal as an input. It then calculates the postulation number, after ensuring that the corresponding ring is Cohen-Macaulay, and prompts the user to enter a point according to the postulation number calculated. The value of the generalized Hilbert-Kunz function at the point is produced as an output and the user is given a choice to enter more points.

\smallskip
\begin{verbatim}
	
HKPolySC = (SCIdeal) -> (
polyRing := ring SCIdeal;

\end{verbatim}

{\bf Step 1}: Check if the Stanley-Reisner ring is Cohen-Macaulay

\begin{verbatim}

if isCM(polyRing/SCIdeal) == false then error "Stanley-Reisner ring is not Cohen-Macaulay";

dimSC := dim (polyRing/SCIdeal);
SComplex := simplicialComplex monomialIdeal SCIdeal;
fvect := fVector(SComplex);
hvect := hVector(SComplex);

\end{verbatim}

{\bf Step 2}: Calculate the derivatives of the polynomial corresponding to the $h$-vector at 1

\begin{verbatim}

Diffh = (i) -> (
  TT := QQ[tt];
  hPoly := sum(0..dimSC, j -> (hvect#j)*(tt^j));
  for j from 1 to i do (
    hPoly = diff(tt, hPoly)
  );
  sub(sub(hPoly, TT/(tt-1)), QQ)
);

\end{verbatim}

Find the list of minimal primes

\begin{verbatim}
	
MinPrimeList := primaryDecomposition SCIdeal;
numPrime := #MinPrimeList;

\end{verbatim}

Ring required for the output polynomial

\begin{verbatim}

OutputRing = QQ[s]; 

\end{verbatim}

Redefining the binomial function

\begin{verbatim}
	
binom = (aa, bb) -> (
  if aa > 0 then return binomial(aa,bb)
  else if (aa == 0 and bb == 0) then return 1
  else return 0
);

\end{verbatim}

{\bf Step 3}: Calculate and print the postulation number 

\begin{verbatim}

PostNum := -position(toList apply(0..dimSC, i-> dimSC-i), i-> hvect#i !=0);
<<"The postulation number is: "<< PostNum <<endl;

\end{verbatim}

{\bf Step 4}: Obtain the point from the user as an input and calculate the Hilbert-Kunz polynomial at that point

\begin{verbatim}

pointer := true; 
while pointer == true do(
  point = read "Enter a number bigger than or equal to the absolute value of the 
  postulation number: ";
  point = value point;    

\end{verbatim}

{\bf Step 5}: The function \texttt{FunctionF} calculates length as in Corollary 3.1. 

\begin{verbatim}

  FunctionF = (QtI, n) -> (									
    dimQt = dim (polyRing/QtI);
    use OutputRing;
    if dimQt == 0 then return 1
    else if dimQt == 1 then return point + n
    else if dimQt == 2 then (
      if n <= point then return point^2 + n^2 + n
      else return (n + point + 1)*(n + point)/2
    )
    else (
      if n <= point then return point^dimQt + dimQt*binom(n+dimQt-1,dimQt)
      else if (point+1 <= n and n <= (dimQt-1)*point-1) then 
        return point^dimQt + sum(1..(dimQt-1), i ->
          ((-1)^(i+1))*binom(dimQt,i)*binom(n-(i-1)*point+dimQt-1,dimQt))
      else return binom(n+point+dimQt-1,dimQt) 
    )
  );

\end{verbatim}

{\bf Step 6}: The function \texttt{AltSumLength} calculates length as in Corollary 3.3

\begin{verbatim}

  AltSumLength = (n) -> (
    polySum = 0;
    for i from 1 to numPrime do (
      CL := subsets(numPrime, i);
      midSum = 0;
      for j from 0 to #CL-1 do (
        midIdeal = sum(0..(i-1), k -> MinPrimeList#(CL#j#k));
        midSum = midSum + FunctionF(midIdeal, n);
      );
      polySum = polySum + (-1)^(i+1)*midSum;
    );
    polySum
  );

\end{verbatim}

{\bf Step 7}: Calculate the Hilbert-Kunz polynomial at the point

\begin{verbatim}
	
  use OutputRing;
  if PostNum == 0 then 
  polyPoint = 2*sum(1..(point-1), n -> AltSumLength(n)) 
    + sum(point..(dimSC*point-1), n -> AltSumLength(n)) 
    - sum(1..((dimSC+1)*point-1), n ->
      sum(0..dimSC,i->(-1)^i*Diffh(i)*(1/i!)*binom(n+dimSC-i-1,dimSC-i))) 
    + 2*sum(0..dimSC, n -> (fvect#(n-1))*(point-1)^n)
  else (
    polyPoint1 = 2*sum(1..(point-1), n -> AltSumLength(n)); 
    if (dimSC == 2) then (polyPoint2 = 0;)
    else (polyPoint2 = sum(point..((dimSC-1)*point-1), n -> AltSumLength(n)););  
    polyPoint3 = sum(1..(dimSC*point-1), n -> 
      sum(0..dimSC, i -> (-1)^i*Diffh(i)*(1/i!)*binom(n+dimSC-i-1,dimSC-i)));
    polyPoint4 = 2*sum(0..dimSC, n -> (fvect#(n-1))*((point-1)^n)); 
    polyPoint = polyPoint1 + polyPoint2 - polyPoint3 + polyPoint4;
  );	
  <<"The value of the Hilbert-Kunz polynomial at the point " << point << " is: " 
  << polyPoint << endl;	

  pointer = read "Do you wish to enter one more point? (true/false): ";
  pointer = value pointer;
  )
)
\end{verbatim}

If the $a$-invariant of the ring is known, the above code can also be used for the non Cohen-Macaulay case with minor modifications.

\bibliographystyle{amsplain}

\end{document}